\documentclass[12pt]{amsart}

\usepackage{amssymb,amsbsy,amsmath,amsfonts,amssymb,amscd}
\usepackage{latexsym}
\usepackage{graphics}
\usepackage{color}
\input xy
\xyoption{all}

\newcommand\sE{{\mathcal E}}

\newcommand\LL{{\mathbb L}}

\newcommand\la{\lambda}

\newcommand\ga{\gamma}
\newcommand\de{\delta}

\DeclareMathOperator{\Pic}{Pic}

\DeclareMathOperator{\Alb}{Alb}
\DeclareMathOperator{\divi}{div}
\DeclareMathOperator{\Tors}{Tors}
\newcommand{\CC}{\ensuremath{\mathbb{C}}}

\newcommand{\ZZ}{\ensuremath{\mathbb{Z}}}

\newcommand{\hol}{\ensuremath{\mathcal{O}}}

\newcommand{\PP}{\ensuremath{\mathbb{P}}}

\newcommand{\ra}{\ensuremath{\rightarrow}}

\def\eea{\end{eqnarray*}}
\def\bea{\begin{eqnarray*}}

\newcommand\dual{\mathrel{\raise3pt\hbox{$\underline{\mathrm{\thinspace d
\thinspace}}$}}}
\newcommand\qe{\ifhmode\unskip\nobreak\fi\quad $\Box$}       

\def\BOX{\hfill\lower.5\baselineskip\hbox{$\Box$}}

\newtheorem{theo}{Theorem}[section]
\newtheorem{remarkk}[theo]{Remark}

\newtheorem{defin}[theo]{Definition}

\newtheorem{lemma}[theo]{Lemma}
\newtheorem{example}[theo]{Example}

\newtheorem{rema}{Remark}[section]
\newtheorem{claim}[theo]{Claim}

\DeclareMathOperator{\Aut}{Aut}

\begin{document}

\title[Burniat surfaces I]{ Burniat surfaces I: fundamental groups
and moduli of primary Burniat surfaces. }
\author{I. Bauer, F. Catanese}

\thanks{The present work took place in the realm of the DFG Forschergruppe 790
"Classification of algebraic surfaces and compact complex manifolds".}

\date{\today}

\maketitle

\section*{Introduction}

In a recent joint paper (\cite{4names}) with Fritz Grunewald and 
Roberto Pignatelli we
constructed many new families of surfaces of general type with $p_g=0$,
hence we got interested about the current status of the classification
of such surfaces, in particular about the structure of their moduli spaces.

For instance, in the course of deciding which families were new and 
which were not new,
we ran into the problem of determining whether surfaces with $K^2 =4$
and with fundamental group equal to the one of Keum-Naie surfaces were
indeed Keum-Naie surfaces. This problem was solved in \cite{keumnaie},
where we showed that any surface homotopically equivalent to a
Keum-Naie surface is a Keum-Naie surface, whence we got a complete description
of a connected irreducible component of the moduli space of surfaces 
of general type.

We soon realized that similar methods would apply to the 'primary' 
Burniat surfaces,
the ones with $K^2 = 6$; hence we got interested about the components 
of the moduli space
containing the Burniat surfaces.

This article is the first  of  a series of  articles devoted to the 
so called Burniat surfaces.
These are several families of surfaces of general type with
$p_g = 0$, $K^2 = 6,5,4,3,2$, first constructed by P. Burniat in
\cite{burniat} as  'bidouble covers' (i.e.,  $(\ZZ / 2 \ZZ)^2$ Galois 
covers) of the plane $\PP^2$
branched on  certain configurations of nine lines.

These surfaces were later considered by Peters in \cite{peters}, who
gave an account of Burniat's construction in the modern language of 
double covers.
He missed however one of the two families with $K^2 = 4$,
the 'non nodal' one.
He also calculated (ibidem) the torsion group $H_1(S, \ZZ)$ for Burniat's
surfaces (observe that a surface of general type with $p_g = 0$ has 
first Betti number $b_1 = 0$).
He asserted  that
$H_1(S,\ZZ) \cong (\ZZ / 2
\ZZ)^{K_S^2}$. This result is  however correct only  for $K^2 \neq 
2$, as we shall see.

  Later, following a suggestion by Miles Reid, another construction of 
these surfaces
was given  by Inoue in
\cite{inoue}, who constructed 'surfaces closely related to Burniat's 
surfaces' with a different
technique as $G^2:=(\ZZ / 2 \ZZ)^3$-quotients of a $G^2$-invariant 
hypersurface $\hat{X}$
of multidegree
$(2,2,2)$ in a product of three elliptic curves.

Another description of the Burniat surfaces as 'singular bidouble 
covers' was later given in
\cite{sbc}, where also other examples were proposed of 'Burniat type 
surfaces'. These however
turn out to give no new examples.

The important feature of the Burniat surfaces $S$ is that their bicanonical map
is a bidouble cover of a normal Del Pezzo surface of degree $K_S^2$ 
(obtained as the anticanonical model
of the blow up of the plane in the points of multiplicity at least 3 of the
divisor given by the union of the lines of the configuration).

Burniat surfaces with $K^2 = 6$ were studied from this point of view
by Mendes-Lopes and Pardini in \cite{mlp}.

Although Burniat surfaces had been known for a long time,
we  found that their most important properties were yet
to be discovered, and we devote  two articles
to show in particular that the four families of Burniat surfaces,
the ones with $ K^2=6,5$ respectively, and the two ones
with $K^2= 4$ (the nodal and the non nodal one)
are irreducible connected components of the moduli space
of surfaces of general type.

Since  there is no reference known where it is proved that
Burniat's surfaces are exactly Inoue's surfaces, we start by  giving
  in the present paper a proof of this fact.

This is crucial in order to calculate the fundamental groups of Burniat's
surfaces with $K^2 = 6,5,4,3,2$. Our proof confirms the results 
stated by Inoue without proof
in his beautiful paper, except for $K^2 = 2$ where
  Inoue's claim turns out to be wrong.

Our proof combines the 'transcendental' description given by Inoue with
delicate algebraic calculations, which are based on  explicit algebraic
normal forms for the 2-torsion of elliptic curves,  described in the 
first section.

We first prove the following:

\begin{theo}
Let $S$ be the minimal model of a Burniat surface.

\noindent
i) $K_S^2 = 6$ $\implies$ $\pi_1(S) = \Gamma$, $H_1(S, \ZZ) =
(\ZZ / 2\ZZ)^6$;

\noindent
ii) $3 \leq K_S^2 \leq 5$ $\implies$ $\pi_1(S) = \mathbb{H} \oplus (\ZZ /
2 \ZZ)^{K^2-2}$,
$H_1(S, \ZZ) = (\ZZ / 2\ZZ)^{K^2}$;

\noindent
iii) $K^2 = 2$ $\implies$ $\pi_1(S) =  H_1(S, \ZZ) = (\ZZ / 2\ZZ)^3$.

\noindent
Here $\mathbb{H}$ denotes the quaternion group of order $8$,
while $\Gamma$ is  a group of affine transformations on $\CC^3$, 
explicitly described in section
\ref{fundgroup}.
\end{theo}

The main result of this article is however the following theorem:

\begin{theo}\label{homotopy1}
    Let $S$ be a smooth complex projective surface which is
homotopically equivalent to a primary Burniat surface.
Then $S$ is a Burniat surface.
\end{theo}

We can then use this result to give an alternative, and less involved 
proof of the following result due to
Mendes-Lopes and Pardini (\cite{mlp}).

\begin{theo}\label{locmod1}
The subset  of the Gieseker moduli space corresponding to
primary Burniat surfaces is an irreducible connected component,
normal, unirational and   of dimension
equal to 4.
\end{theo}

In \cite{mlp} openness is shown (using Burniat's description of a 
primary Burniat surface)
by standard local deformation theory of bidouble covers. We  give an 
alternative proof of this
result  using Inoue's
description.

For
the closedness, Mendes Lopes and Pardini use
their characterization of primary Burniat surfaces as exactly those surfaces
with $p_g = 0$, $K^2 = 6$ such that
the bicanonical map has degree $4$.

Our proof is much less involved. It only uses the description of the
fundamental group as an affine group of
transformations of $\CC^3$.

In the second article we shall show that the Burniat surfaces with
$K^2 = 5$ yield  an irreducible connected
component of dimension $3$ in the moduli space of surfaces of general
type. Instead,  there are two
different configurations in the plane giving Burniat surfaces with
$K^2=4$. We shall show that, in fact, both yield
an irreducible connected component of dimension $2$ in the moduli
space of surfaces of general type.

This is interesting, since it follows that the bicanonical map of $S$
is a bidouble cover of a Del Pezzo surface of degree $K^2_S$
for all the surfaces in the connected component.

There is only one Burniat surface with $K^2 = 2$, and since its 
fundamental group is
$(\ZZ / 2\ZZ)^3$, it turns out to be a surface in the 6 dimensional family of
standard Campedelli surfaces (\cite{miyaoka}), for which the bicanonical
map is then a degree 8 covering of the plane.

We analyse it briefly in the last section.
In fact, at the moment of completing the paper we became aware of the 
article \cite{kulikov}
where the author had already pointed out and corrected the errors of 
\cite{inoue} and \cite{peters}
  on the fundamental group and the first homology of
the Burniat surface with $K^2 = 2$.

\section{The Legendre and other normal forms for 2-torsion of elliptic curves
}\label{legendre}

This section reviews  classical mathematics which will be
reiteratedly used in the sequel.

The Legendre form of an elliptic curve is given by an equation of the form

$$ y^2 = (\xi^2 -1) ( \xi^2 - a^2) .$$

It yields a curve $\sE'$ of genus $1$ as a double  cover of $\PP^1$ branched on
the 4 points $  \xi = \pm 1$, $ \xi = \pm a$.

These 4 points yield 4 points on $\sE'$ , $P'_{1},  P'_{-1}, P'_{a},
P'_{-a},  $ which correspond to the
2-torsion points, once any of them is fixed as the origin, as we
shall more amply now illustrate.

We consider  now 3 automorphism of order 2 of $\sE'$ defined by:
$$ g'_1 (\xi,y) : =  (-\xi, -y), \ g'_2 (\xi,y) : =  (\xi, - y), \ 
g'_3 (\xi,y) :
=  (-\xi, y).$$
We get in this way an action of $(\ZZ / 2 \ZZ)^2$ on $\sE'$ such that
the quotient is $\PP^1$,
with coordinate $ u : = \xi^2 $.

Clearly the quotient by $g'_2$ is the original $\PP^1$ with
coordinate $\xi$, hence
$g'_2$ corresponds to multiplication by $-1$ on the elliptic curve,
once we fix one of the above points as the origin.

    The quotient of $\sE'$ by $g'_3$ is
instead  the smooth curve of genus $0$, given by the conic $ y^2 =
(u-1)(u-a^2)$.

What is more interesting is the quotient of $\sE'$ by $g'_1$: the
invariants are $u$ and $r : = \xi y$,
thus we obtain as quotient the elliptic curve
$$  \sE : = \{ r^2 = u (u-1) (u-a^2) \}$$
in Weierstrass normal form.

This shows that $g'_1$ is the translation by a 2-torsion element $\eta'$.

By looking at the action of
$g'_1$ on the 4 above points, we see that $\eta'$ is the class of
the degree zero divisor $ [ P'_{1}]-  [P'_{-1}]$.

In other words, the
divisor classes of degree $0$
$$ \eta'  : = [ P'_{1}]-  [P'_{-1}] =  [P'_{a}] - [  P'_{-a}] , \
\eta '' : = [ P'_{1}]-  [P'_{a}] =  [P'_{-1}] - [  P'_{-a}] $$
generate $Pic^0 (\sE')[2] \cong (\ZZ / 2 \ZZ)^2$.

We can now also understand that the automorphism $g'_3$ , which is
the product $g'_1 g'_2 = g'_2 g'_1$,
has as fixed points the 4 points lying over $ \xi = 0$ and $ \xi=
\infty$. These correspond to the 4-torsion points
whose associated translation has as square the translation by the
2-torsion point $\eta'$.

More important, we want to give now a nicer form for the group of
translations of order 2 of an elliptic curve
(this leads  to the theory of 2-descent on an elliptic curve).
This form will be used in the sequel, but in another coordinate system
for the line $\PP^1$ with
coordinate $\xi$.

To this purpose, let us consider the curve $C$ defined by

$$  v^2 = (\xi^2 -1) , \ w^2 =  (\xi^2 -a^2).$$

We shall show that this curve is the same elliptic curve $\sE$ which
we had above.

In fact, setting $ y : = vw $, we see that we obtain $C$ as a double cover
of $\sE'$, which is unramified
(as we see by calculating the ramification of the $(\ZZ / 2
\ZZ)^2$-Galois cover of
$\PP^1$ with coordinate $\xi$).

The transformations of order 2
$$ g_1 : (\xi,v,w) \mapsto (\xi, -v, -w), \ f_2 : (\xi,v,w) \mapsto 
(-\xi, v, -w),$$
$$ f_3 : (\xi,v,w) \mapsto (-\xi, -v, w)$$
generate a group $ H  \cong (\ZZ / 2 \ZZ)^2$ such that the quotient
curve is the elliptic curve
$\sE$, since the invariants are $ \xi^2 = u, \ v^2 = u-1, \  w^2 = u - a^2,
  \ \xi vw= \xi y = r$.

The  quotients  by the above involutions are respectively  $\sE'$,
the elliptic curve
$$ \sE'' : = \{ t^2 =  (v^2 + 1) ( v^2 + 1 - a^2) \}  $$
(where we have set $ t : = \xi w$), and the elliptic curve
$$ \sE''' : = \{ s^2 =  (w^2 + a^2) ( w^2 +  a^2 - 1) \}  $$
(where we have set $ s : = \xi v$).

The first conclusion is that $C$ is isomorphic to $\sE$, the group
$H$ is the group of translations by the
2-torsion points of $\sE$, whereas the quotient map $ C \ra \sE = C/H$ is
multiplication by 2 in the elliptic curve $C
\cong \sE$.

We have another group $G \cong (\ZZ / 2 \ZZ)^2$ acting on $C \cong \sE$, namely
the one with quotient  the $\PP^1$ with coordinate $\xi$. Here, we set
$$ g_2 :  (\xi ,v,w) \mapsto (\xi, -v, w), \ g_3 : = g_1 g_2 = g_2 g_1 :
(\xi,v,w) \mapsto (\xi, v, -w).$$

Again, $g_1$ corresponds to translation by a 2-torsion element
$\eta$, while we view $g_2$
as multiplication by $ -1$. The fixed points of $g_2$ are the points
with $v=0$,
i.e., the 4 points with $ v=0, \ \xi = \pm 1, \  w = \pm \sqrt {1 - a^2}$.

Translation by $\eta$ then acts on them  simply by multiplying their
$w$ coordinate by $-1$.

An important observation is that the covering $ C \ra \PP^1$,
where $\PP^1$ has coordinate $u$, is a
$(\ZZ / 2 \ZZ)^3$-Galois cover of
the $\PP^1$ with coordinate $u$ which is the maximal Galois covering of $\PP^1$
branched on the
4 points $0,1,a^2, \infty$ and with group of the form $(\ZZ / 2 \ZZ)^m$.

It would be nice if  also for surfaces one could treat such Galois
covers with group $(\ZZ / 2 \ZZ)^m$
in the same elementary way . This however can be done only in the
birational setting, since in dimension $\geq 2$ we have different normal models
for the same function field.
Hence we have to resort to the theory of abelian covers, developed in
\cite{ms}, \cite{pardini}, \cite{sbc}.

In this biregular theory, coverings are described through equations holding
in certain vector bundles.
To compare the surface case with the curve case it is therefore 
useful first of all
to rewrite the above $(\ZZ / 2 \ZZ)^3$-Galois cover in terms of 
homogeneous coordinates.

And, for later calculations, it will be convenient to replace the 
points $\xi = \pm 1$ with the
points $0, \infty$.

We replace then the affine coordinates $( 1 : \xi)$  by coordinates $(x' : x)$
with $ \frac{x}{x'} = \frac{\xi - 1}{\xi + 1} $.

We can then rewrite the $(\ZZ / 2 \ZZ)^2$ cover as the normalization 
of the curve
in $\PP^1 \times \PP^1 \times \PP^1$ given by
$$ \{ (v' : v),  (w' : w),(x' : x) \ | \ v^2 x' = {v'}^2 x , \ w^2 
{x'}^2 = {w'}^2
( x^2 - x x' (b + \frac{1}{b} ) + {x'}^2)\}.$$

Now the involution exchanging pairs of branch points is simply the involution
$ (x' : x) \ra (x : x' )$.

The normalization is obtained simply by considering the curve of genus 1
which is the subvariety of the vector bundle whose sheaf of sections
on $\PP^1$ is $\hol_{\PP^1} (1) \oplus \hol_{\PP^1} (1) $, given by equations
$$ V^2 = x x' , \ W^2 = ( x^2 - x x' (b + \frac{1}{b} ) + {x'}^2),$$
which is shorthand notation for the following two equations
in the local chart outside $ x' = 0$, respectively in the local chart 
outside $ x = 0$:

$$  {(\frac{v}{v'})}^2  =  (\frac{x}{x'}) ,  \  {(\frac{w}{w'})}^2 =
{(\frac{x}{x'})}^2 - (b + \frac{1}{b} ) \frac{x}{x'}  + 1 ,$$
$$  {(\frac{v'}{v})}^2  =  (\frac{x'}{x}) , \ {(\frac{w}{w'}\frac{x'}{x})}^2 =
{(\frac{x'}{x})}^2 - (b + \frac{1}{b} ) \frac{x'}{x}  + 1  .$$

In other words, we have $ v^2  = x, \ {v' }^2 = x' $, hence $ V = v v'$.
While, setting $ W : = {(\frac{w}{w'})} x'$, we get
$ W^2 = ( x^2 - x x' (b + \frac{1}{b} ) + {x'}^2)$.

We have now the group $(\ZZ/2\ZZ)^3$ acting on $C$
by the following transformations
$$ g_1 : ((x':x),(v':v), (w':w)) \mapsto ((x':x),(v':-v),(w':-w)), $$
$$ f_2 : ((x':x),(v':v), (w':w)) \mapsto ((x:x'),(v:v'),(w' x : - w x')),$$
$$ f_3 : ((x':x),(v':v), (w':w)) \mapsto ((x:x'),(-v:v'),(w' x : w x')),$$
$$ g_2: ((x':x),(v':v), (w':w)) \mapsto ((x':x),(v':-v),(w':w)),$$
$$ g_3 =g_1 g_2: ((x':x),(v':v), (w':w)) \mapsto ((x':x),(v':v),(w':-w)).$$

The sections $V$ and $W$ are clearly eigenvectors for the group action.
It is easy to see , in view of the above table, that the image of  $ V = v v'$
equals $ - V,  V, -V, - V, V $ respectively,
while the image of $W$ equals
$- W , - W, W, W, -W$ respectively.

\section{Burniat surfaces are Inoue surfaces}\label{burniateqinoue}

The aim of this section is to show that Burniat surfaces are Inoue surfaces.
This fact seems to be  known to the experts, but, since we did
not find any reference,
we shall
provide a proof of this assertion, which is indeed crucial for our main result.

In \cite{burniat}, P. Burniat constructed a series of families of
surfaces of general type with $K^2 =6,5,4,3,2$
and $p_g = 0$ (of respective dimensions $4,3,2,1,0$) as singular
bidouble covers
(Galois covers with group $(\ZZ/2\ZZ)^2$) of the projective plane
branched on 9 lines.
    We briefly recall the construction.

\medskip
\noindent
Let $P_1, P_2, P_3 \in \PP^2$ be three non collinear points  and
denote by $Y:=\hat{\PP}^2(P_1, P_2,P_3)$
    the blow up of $\PP^2$ in $P_1, P_2, P_3$.

  $Y$ is a Del Pezzo
surface of degree $6$
and it is the closure of the graph of the rational map
$$\epsilon: \PP^2 \dashrightarrow \PP^1 \times \PP^1 \times \PP^1$$
such that $$\epsilon (y_1 : y_2: y_3)  = ((y_2 : y_3) ( y_3: y_1) (
y_1: y_2)).$$

It is immediate to observe that $Y \subset \PP^1 \times \PP^1 \times \PP^1$
is the hypersurface of type $(1,1,1)$:
$$ Y = \{(( x_1' : x_1), (  x_2':  x_2), ( x_3':  x_3)) \ | \ x_1 x_2 x_3
= x_1' x_2' x_3' \}.$$

\begin{lemma}\label{splitting}
Consider the cartesian diagram
\begin{equation*}
\xymatrix{
p^{-1}(Y) \ar[r]^{p} \ar[d]&Y \ar[d]^i\\
\PP^1 \times \PP^1 \times \PP^1\ar[r]_{p} & \PP^1 \times \PP^1 \times \PP^1\\
}
\end{equation*}
where $p: \PP^1 \times \PP^1 \times \PP^1 \ra \PP^1 \times \PP^1
\times \PP^1$ is the $(\ZZ/ 2 \ZZ)^3$ -Galois
covering given by $x_i = v_i^2, \ x_i' = (v_i')^2$. Then $p^{-1}(Y)$
splits as the union
$p^{-1}(Y) = Z \cup Z'$ of  two
degree 6  Del Pezzo surfaces, where
$$
Z:= \{((v_1:v'_1),(v_2:v'_2),(v_3:v'_3)) : v_1v_2v_3 = v'_1v'_2v'_3 \}
$$
and
$$
Z':= \{((v_1:v'_1),(v_2:v'_2),(v_3:v'_3)) : v_1v_2v_3 = -v'_1v'_2v'_3 \}.
$$
And $p|Z$ induces on $\PP^2$ the Fermat squaring map
$$
(y_0 : y_1:
y_2) \mapsto (y_0^2 : y_1^2: y_2^2).$$

Moreover, $Z \cap Z' = \{v_1v_2v_3 = v'_1v'_2v'_3 = 0 \}$, which is
the union of 6 lines
yielding in each Del Pezzo surface the fundamental hexagon of
the blow up of
$\PP^2$ in three non collinear points (i.e., the pull back of the
triangle with vertices the three points).
\end{lemma}

\begin{proof}
The equation of $p^{-1}(Y) $ is $x_1 x_2 x_3 = x_1' x_2' x_3' $, i.e.,
$$ (v_1v_2v_3 )^2 = (v'_1v'_2v'_3)^2.$$

\noindent
The surface $Z$ is invariant under the subgroup
$$
  G^o \subset \{\ 1,
-1\}^3 \cong (\ZZ/ 2 \ZZ)^3, \
G^o \cong (\ZZ/ 2 \ZZ)^2,
$$
$$
  G^o = \{ (\epsilon_i) \in  \{\pm 1\}^3
| \prod_i \epsilon_i = 1\}.
$$

$ G^o $ acts on $Y$ by sending $ v_i \mapsto \epsilon_i v_i$, $v'_i
\mapsto  v'_i$,
and this is easily seen to give on $\PP^2$ the Galois group of the
Fermat squaring map.

\end{proof}

We denote by $E_i$ the exceptional curve
lying over $P_i$  and by $D_{i,1} $ the unique effective divisor in $
|L - E_i- E_{i+1}|$,
i.e., the proper transform of the line $y_{i-1} = 0$, side of the
triangle joining the points $P_i, P_{i+1}$.

For the present choice of coordinates $E_i$ is the side $ x_{i-1} =
x'_{i+1} = 0 $ of the hexagon,
while $D_{i,1} $ is the side $ x_{i} = x'_{i+1} = 0 $ of the hexagon.

Consider on $Y$ the following divisors
$$ D_i = D_{i,1} + D_{i,2} + D_{i,3} + E_{i+2} \in |3L - 3E_i -
E_{i+1}+E_{i+2}|,$$

where $D_{i,j} \in |L - E_i|, \ \rm{for} \ j = 2,3, \  D_{i,j} \neq
D_{i,1}$, is the proper transform of
another  line through
$P_i$ and $D_{i,1} \in |L - E_i- E_{i+1}|$
is as above. Assume also that all the corresponding lines in $\PP^2$
are distinct,
so that $D : = \sum_i D_i$ is  a reduced divisor.

Observe that all the indices in  $ \{1,2,3\}$ have here to be understood as
residue classes modulo 3.

Note that, if we define  the divisor $\mathcal{L}_i : = 3L - 2
E_{i-1} - E_{i+1}$, then
$$D_{i-1} + D_{i+1} = 6L - 4 E_{i-1} - 2E_{i+1} \equiv 2
\mathcal{L}_i,$$ and we can consider
(cf. \cite{sbc}) the associated bidouble cover $X \rightarrow Y$
branched on $D : = \sum_i D_i$
(but with different ordering of the indices: we take here one which is more apt
for our notation).

We recall that this precisely means the following: let $D_i =
\divi(\delta_i)$, and let $u_i$ be
    a fibre coordinate of the geometric line bundle $\LL_i$, whose sheaf
of holomorphic sections is
$\hol_Y(\mathcal{L}_i)$.

Then $X \subset \LL_1 \oplus \LL_2 \oplus
\LL_3$ is given by the equations:
$$
u_1u_2 = \delta_1 u_3, \ \ u_1^2 = \delta_3 \delta_1;
$$
$$
u_2u_3 = \delta_2 u_1, \ \ u_2^2 = \delta_1 \delta_2;
$$
$$
u_3 u_1 = \delta_3 u_2, \ \ u_3^2 = \delta_2 \delta_3.
$$

   From the birational point of view, we are simply adjoining to the
function field of $\PP^2$
two square roots, namely $\sqrt  \frac{\Delta_1}{ \Delta_3}$ and
$\sqrt  \frac{\Delta_2}{ \Delta_3}$,
where $\Delta_i$ is the cubic polynomial in $\CC[x_0,x_1,x_2]$ whose
zero set has $D_i$ as strict transform.

This shows clearly that we
have a Galois cover with group $(\ZZ / 2 \ZZ)^2$.

The equations above give a biregular model $X$ which is nonsingular
exactly if the divisor $D$
does not have points of multiplicity 3 (there cannot be points of
higher multiplicities). These points
give then quotient singularities of type $\frac{1}{4} (1,1) $, i.e.,
the quotient of $\CC^2$
by the action of $(\ZZ / 4 \ZZ)$ sending $ (u,v) \mapsto (iu, iv)$
(or, equivalently , the affine cone
over the 4-th Veronese embedding of $\PP^1$).

This (cf. \cite{cime} for more details) can be seen  by an elementary
calculation.

Assume in fact that $\delta_1, \delta_2, \delta_3$ are given in local
holomorphic coordinates
by $x,y, x-y$, and that we define locally $w_i$ as the square root of
$\delta_i$. Then:
$$ w_1^2 = x, \  w_2^2 = y,\  w_3^2 = x-y \ \Rightarrow \  w_3^2 = 
w_1^2 - w_2^2 .$$

Therefore the singularity is an $A_1$ singularity, quotient of $\CC^2$
by the action of $(\ZZ / 2 \ZZ)$ sending $ (u,v) \mapsto (-u, -v)$ (here,
$w_3 = uv, \ u^2 = w_1 + w_2, \ v^2 =  w_1 - w_2$). The action of 
$(\ZZ / 4 \ZZ)$
on $\CC^2$ induces the action of  $(\ZZ / 2 \ZZ)$ on the $A_1$ singularity,
sending $w_i \mapsto - w_i, \ \forall i$. Finally, the functions $u_i
= w_{i+1} w_{i+2}$
and $w_i^2 = \delta_i$ generate the  $(\ZZ / 4 \ZZ)$-invariants, subject to the
linear relation $\delta_1 - \delta_2 = \delta_3$.

The singularity can be resolved by blowing up the point $x=y=0$, and
then the inverse image of
the exceptional line is a smooth rational curve with self intersection $-4$.

\begin{defin}
    A {\em primary Burniat surface} is a surface constructed as above,
and which is moreover smooth.
    It is then a minimal surface $S$ with $K_S$ ample, and with $K_S^2 =
6$, $p_g(S) =q(S)= 0$.

     A {\em secondary Burniat surface} is a surface constructed as
above, and which  moreover has $ 1
     \leq m \leq 2$ singular points (necessarily of the type described above).
    Its minimal resolution is then a minimal surface $S$ with $K_S$ nef
and big, and with $K_S^2 = 6-m$, $p_g(S) =q(S)= 0$.

    A {\em tertiary Burniat surface} is a surface constructed as above,
and which  moreover has $ 3
     \leq m \leq 4$ singular points (necessarily of the type described above).
    Its minimal resolution is then a minimal surface $S$ with $K_S$ nef
and big, and with $K_S^2 = 6-m$, $p_g(S) =q(S)= 0$.

\end{defin}

\begin{rema}
1) We remark that for $K_S^2 =4$ there are two possible type of
configurations. The  one
where there are three collinear points of multiplicity at least 3 for
the plane curve formed by the 9 lines
leads to a Burniat surface $S$ which we call of  {\em nodal type},
and with $K_S$ not ample,
since the inverse image of the line joining the 3 collinear points is
a (-2)-curve (a smooth rational curve of
self intersection $-2$).

  In the other cases with $K_S^2 =4, 5$,
instead, $K_S$ is ample.

2) In the nodal case,  if we  blow up  the two $(1,1,1)$ points of
$D$, we obtain a weak Del Pezzo surface, since it
contains a (-2)-curve. Its anticanonical model has a node (an
$A_1$-singularity, corresponding to the contraction of
the (-2)-curve). In the non nodal case, we obtain a smooth Del Pezzo
of degree $4$.

This fact has obviously been overlooked by \cite{peters}, since he
only mentions the nodal case.

\noindent
In the sequel to this paper we shall show that in the case of
secondary Burniat surfaces  with $K^2_S = 4$ these
two families indeed give two different connected components of
dimension $2$ in the moduli space. And also that
secondary Burniat surfaces  with $K^2_S = 5$ form a connected
component of dimension $3$ in the moduli space.

\noindent
3) We illustrate the possible configurations in the plane in figure
\ref{configs}.
\end{rema}
\begin{figure}[htbp]\label{configs}
\begin{center}
\scalebox{0.7}{\includegraphics{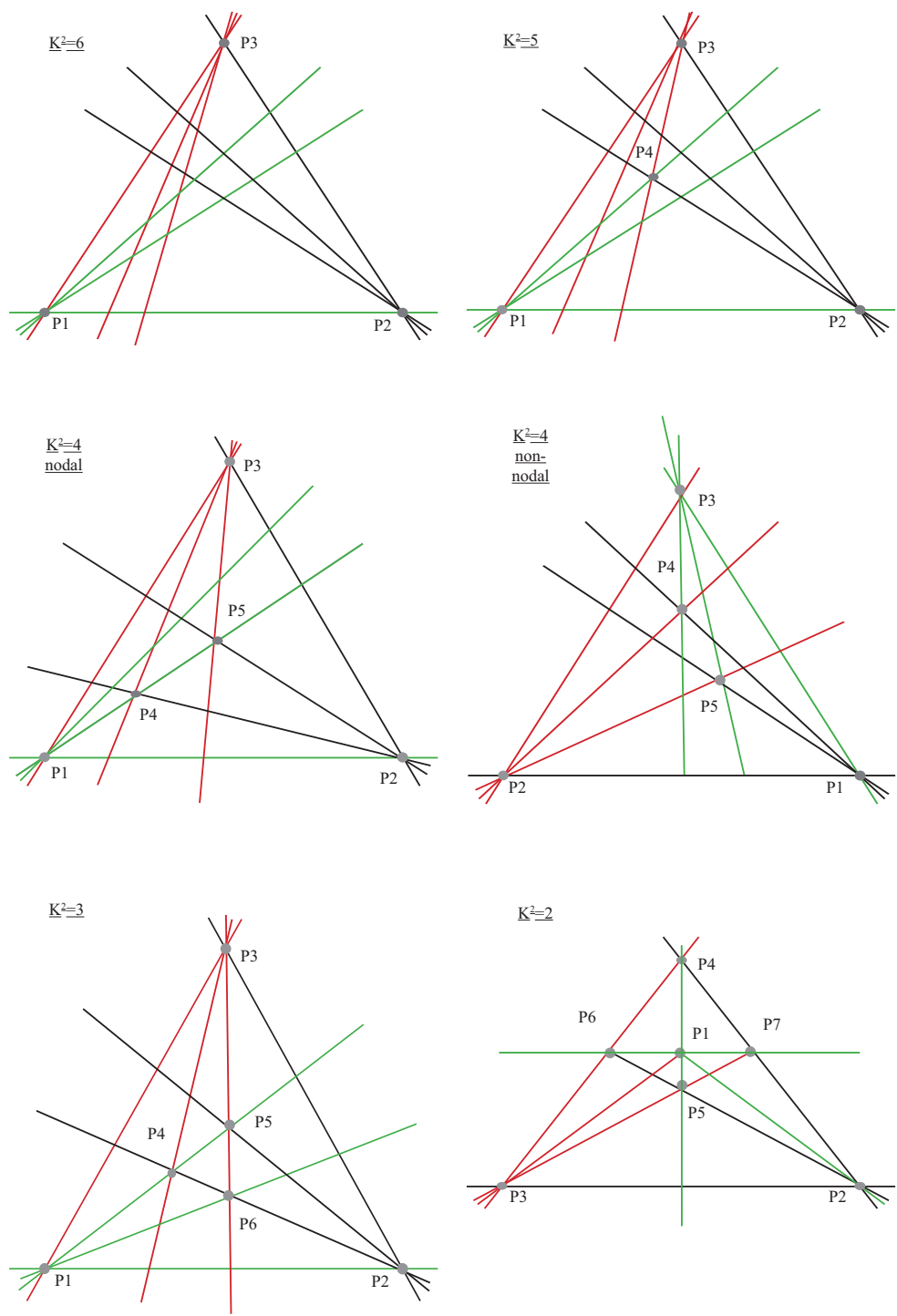}}
\end{center}
\caption{Configurations of lines}
\end{figure}

In \cite{inoue} Inoue constructed a series of  families of surfaces
with $K^2 = 6,5,4,3,2$ and $p_g = 0$
(of respective dimensions $4,3,2,1,0$, exactly as for the Burniat
surfaces) as the $(\ZZ / 2 \ZZ)^3$ quotient of an
invariant hypersurface of type $(2,2,2)$ in a product of three
elliptic curves. As already mentioned, it seems to be  known to the specialists
that these Inoue's  surfaces are exactly the Burniat's surfaces, but
for lack of a reference we show here:

\begin{theo}
Burniat's surfaces are exactly Inoue's surfaces.
\end{theo}

\begin{proof}
Consider as in lemma \ref{splitting} the cartesian diagram
\begin{equation*}
\xymatrix{
p^{-1}(Y) \ar[r]^{p} \ar[d]&Y \ar[d]^i\\
\PP^1 \times \PP^1 \times \PP^1\ar[r]_{p} & \PP^1 \times \PP^1 \times \PP^1\\
}
\end{equation*}
where $p: \PP^1 \times \PP^1 \times \PP^1 \ra \PP^1 \times \PP^1
\times \PP^1$ is the $(\ZZ/ 2 \ZZ)^3$ -Galois
covering given by $x_i = v_i^2, \ x_i' = (v_i')^2$. Then $p^{-1}(Y)$
splits as the union of  two
degree 6  Del Pezzo surfaces
$p^{-1}(Y) = Z \cup Z'$, where
$$
Z:= \{((v_1:v'_1),(v_2:v'_2),(v_3:v'_3)) : v_1v_2v_3 =  v'_1v'_2v'_3 \}
$$
and
$$
Z':= \{((v_1:v'_1),(v_2:v'_2),(v_3:v'_3)) : v_1v_2v_3 = -  v'_1v'_2v'_3 \}.
$$
Recall that the subgroup of $( \ZZ / 2\ZZ)^3$ stabilizing $Z$ is $G^o
= \{ (\epsilon_i) \in  \{\pm 1\}^3 | \prod_i
\epsilon_i = 1\}$.

We can further extend the previous diagram by considering a $( \ZZ /
2\ZZ)^6$ Galois-covering
    $\hat{p}: \sE_1 \times \sE_2 \times \sE_3 \ra \PP^1 \times \PP^1
\times \PP^1$ obtained by taking,
    with  different
choices of the $(x' : x)$ coordinates, the direct product of three $( \ZZ /
2\ZZ)^2$ Galois-coverings $\sE_i \ra  \PP^1$ as in
section
\ref{legendre}.

What we have now explained is summarized in the bottom two lines of 
the following
commutative diagram,
where  $\hat{X} $ is defined as the inverse image of the Del Pezzo surface $Z$.

\begin{equation*}
\xymatrix{
\hat{X} \ar[r]^{G^2} \ar[d]_{\hat{i}}& X = \hat{X} /G^2\ar[dr] & \\
\hat{X } \cup \hat{X}' \ar[r] \ar[d] & Z \cup Z' \ar[d] \ar[r] &Y \ar[d]\\
\sE_1 \times \sE_2 \times \sE_3 \ar[r]^{(\ZZ /2)^3}&\PP^1 \times
\PP^1 \times \PP^1 \ar[r]^{(\ZZ/2)^3}&\PP^1
\times
\PP^1
\times
\PP^1. }
\end{equation*}
Note that the vertical map $i: \hat{X} \cup \hat{X}' \hookrightarrow
\sE_1 \times \sE_2 \times \sE_3$ is the
inclusion   of $\hat{X} \cup \hat{X}'$ as a divisor of multidegree
$(4,4,4)$ splitting as a union of two
divisors of respective multidegrees $(2,2,2)$.

Next we want to show that $\hat{X}$ is a $( \ZZ / 2\ZZ)^3$ Galois
covering of $X$,
ramified only in the points of type
$\frac{1}{4}(1,1)$ (and hence \'etale in the case of a primary Burniat $X$).

In fact, the stabilizer of $\hat{X}$  is $$G^1 :=
    \{(\epsilon_i, \epsilon'_i) \in
\{\pm 1\}^3
\times
\{\pm 1\}^3 |
\prod_i
\epsilon_i = 1\} \cong (\ZZ /2 \ZZ)^5.$$

The action of $G^1$ makes $\hat{X}$ a $( \ZZ / 2\ZZ)^5$ Galois covering of $Y$,
and we claim that we obtain $X$ as an intermediate cover by
setting
$$
u_i = W_{i-1}W_{i}v_{i}v'_{i}.
$$

Let us  denote $(D_{i,2} + D_{i,3})$ by $D'_i$. This is the divisor
defined by a
section $\de'_i = 0$ which is the pull back of a homogeneous
polynomial of degree 2
on the i-th copy of $\PP^1$  (this polynomial is the polynomial
$ ( x_i^2 - x_i x_i' (b_i + \frac{1}{b_i} ) + {x_i'}^2)$
in the notation of section \ref{legendre}).

Let us then write
$$
D_i =  (D_{i,1} + E_{i+2}) + (D_{i,2} + D_{i,3}) =  D_{i,1} + E_{i+2} + D'_i.
$$

Observe that $ \divi (x_i) = D_{i,1} +  E_{i+1}$,
    $ \divi (x'_i) = D_{i-1,1} +  E_{i-1}$, whence
$$D_i  + D_{i-1}  = D'_i  + D'_{i-1} + D_{i,1} + E_{i+2} + D_{i-1,1} + E_{i+1}=
\divi (\de'_i \de'_{i-1} x_i x'_i).$$

Now, the $(\ZZ / 2 \ZZ)^2$  Galois- covering of the i-th copy of
$\PP^1$ is given by:

$$
(v'_i v_i) ^2 =  x_i x'_i ,  \ W_i^2 =\de'_i.
$$

Since
$$
u_i^2 =  \de_{i} \de_{i-1},
$$

we see that
$$u_i^2  = \de_{i} \de_{i-1} = \de'_i \de'_{i-1} x_i x'_i = (W_{i}
W_{i-1} v_{i}v'_{i})^2.
$$

Whence we have established our claim that setting
$$
u_i = W_{i-1}W_{i}v_{i}v'_{i}
$$
we get a mapping $\hat{X} \ra X$.

We also see that $\hat{X} \ra X$ is Galois with  Galois group  the
subgroup $G^2 < G^1$ leaving
each $u_i $
invariant, which, by the above formulae,  is given by
$$
\{(\epsilon_i, \epsilon'_i) \in G^1 |
\epsilon_{i-1}' \epsilon_{i}'\epsilon_{i} = 1, i \in \{1,2,3\}\}
\cong (\ZZ /2 \ZZ)^3.
$$

The last isomorphism follows since the $\epsilon_i'$'s  determine
$\epsilon_i = \epsilon_{i-1}' \epsilon_{i}'$.

A natural basis for $G^2 \leq G^1 \leq (\ZZ / 2 \ZZ)^6 \cong (\ZZ / 2
\ZZ)^3 \oplus (\ZZ / 2 \ZZ)^3$ is given by
$$
    (\begin{pmatrix}
    1\\ 0\\ 0
\end{pmatrix}, \begin{pmatrix}
     1\\ 1\\ 0
\end{pmatrix}) =:g_1, \
(\begin{pmatrix}
    0\\ 1\\ 0
\end{pmatrix}, \begin{pmatrix}
     0\\ 1\\ 1
\end{pmatrix})=:g_2, \ (\begin{pmatrix}
    0\\ 0\\ 1
\end{pmatrix}, \begin{pmatrix}
     1\\ 0\\ 1
\end{pmatrix}) =:g_3.
$$
Therefore, if $z_i$ is a uniformizing parameter for the elliptic curve
$\sE_i$, with $z_i =0$ corresponding
to the origin of  $\sE_i$, we see
that the action of
$G^2$  on
$\sE_1
\times
\sE_2
\times
\sE_3$ (cf. section
\ref{legendre}) is given as follows:
$$
g_1\begin{pmatrix}
z_1 \\
z_2 \\
z_3
\end{pmatrix}
=
\begin{pmatrix}
z_1 +\eta_1 \\
-z_2 \\
z_3
\end{pmatrix}, \ g_2\begin{pmatrix}
z_1 \\
z_2 \\
z_3
\end{pmatrix}
=
\begin{pmatrix}
z_1  \\
z_2 + \eta_2 \\
-z_3
\end{pmatrix}, \ g_2\begin{pmatrix}
z_1 \\
z_2 \\
z_3
\end{pmatrix}
=
\begin{pmatrix}
- z_1  \\
z_2  \\
z_3 + \eta_3
\end{pmatrix}.
$$

\begin{rema}
    If $X$ is a primary Burniat surface, then $\hat{X} \rightarrow X$ is
an \'etale $(\ZZ / 2 \ZZ)^3$- covering.

Instead, for
each $(1,1,1)$ - point of $D = D_1 + D_2 + D_3$, $X$ has a singular
point of type $\frac{1}{4}(1,1)$, and $\hat{X}
\rightarrow X$ is ramified in exactly these singular points, yielding
$4$ nodes on $\hat{X}$ for
each one of these singular
points on $X$.
\end{rema}

Since $\hat{X} $ is a divisor of type $(2,2,2)$ invariant by the
action of $G^2$,
we have seen that  any Burniat surface $X$ is an Inoue surface.

Conversely, assume that $X = \hat{X}/ G^2  $ is an Inoue surface: 
since every such $G^2$-
invariant surface $\hat{X}$ is the pull back of a Del Pezzo
surface
\begin{equation}\label{constant}
Z_c : = \{ (v'_i, v_i) | v_1v_2v_3 - c v'_1v'_2v'_3 = 0\},
\end{equation}
we see that $X$ is a Burniat surface.

\end{proof}

\begin{rema}
In the above equation (\ref{constant}) there is a constant $c$ 
appearing, whereas in the
previous description we had normalized this constant to be equal to $1$.

On each $\PP^1$ there are the points $v_i=0$, $v'_i=0$, hence these
coordinates are determined up to a constant $\la_i$. In turn, we have two more
branch points, forming the locus of zeroes of an equation which we normalized
as being $ v_i^2 + (b_i + \frac{1}{b_i}) \ v_i v'_i + {v'_i}^2 = 0$.
This normalization now determines the constant $\la_i$ uniquely,
and finally with these choice of coordinates
we get the equation (\ref{constant}) with $ c = \prod_i \la_i$,
and we see that $c$ is a function of $b_1, b_2, b_3$.
\end{rema}

\section{The fundamental groups of Burniat surfaces}\label{fundgroup}

The aim of his section is to combine our and  Inoue's representation of Burniat
surfaces in order to calculate the fundamental
groups of the Burniat surfaces with $K^2 = 6,5,4,3,2$.

In \cite{inoue} the author gave a table of the respective
fundamental groups, but without supplying a proof.
As we shall now see, his assertion is right for  $K^2 = 6,5,4,3$ but 
wrong for the case $K^2 = 2$.
So we believe it worthwhile to give a detailed proof, especially in 
order to cast away any doubt
on the validity of his assertion for  $K^2 = 6,5,4,3$.

Let $(\sE,o)$ be any elliptic curve, and consider as in section
\ref{legendre} the
$G = (\ZZ /2
\ZZ)^2 =
\{0, g_1, g_2, g_3:=g_1g_2
\}$ - action given by
$$
g_1(z) := z + \eta, \ \ g_2(z) = -z,
$$
where $\eta \in \sE$ is a $2$ - torsion point of $\sE$.
\begin{rema}\label{invdiv}
    The divisor $[o] + [\eta] \in Div^2(\sE)$ is invariant under $G$,
hence the invertible
sheaf  $\hol_{\sE}([o] + [\eta])$ carries a natural $G$-linearization.

In particular, $G$ acts on the vector space $ H^0(\sE, \hol_{\sE}([o] +
[\eta]))$ which splits then as a direct
sum $$H^0(\sE, \hol_{\sE}([o] + [\eta])) = \bigoplus_{\chi \in G^*} H^0(\sE,
\hol_{\sE}([o] + [\eta]))^{\chi}$$
of the eigenspaces corresponding to the characters  $\chi$  of $G$.

We shall use a self explanatory notation:

for instance
$H^0(\sE, \hol_{\sE}([o] + [\eta]))^{+-}$ is  the eigenspace corresponding
to the character
$\chi$ such that $\chi(g_1) = 1$, $\chi(g_2) = -1$.

\end{rema}

We recall the following:

\begin{lemma}[\cite{keumnaie}, lemma 2.1]\label{h++}
Let $\sE$ be as above.  Then
    $$H^0(\sE, \hol_{\sE}([o] + [\eta])) = H^0(\sE, \hol_{\sE}([o] +
[\eta]))^{++} \oplus H^0(\sE, \hol_{\sE}([o] +
[\eta]))^{--}$$

\noindent
I.e., $H^0(\sE, \hol_{\sE}([o] + [\eta]))^{+-} = H^0(\sE,
\hol_{\sE}([o] + [\eta]))^{-+} =0$.
\end{lemma}

Let now $\sE_i := \CC / \Lambda_i$, $i=1,2,3$,
be three complex elliptic
curves, and write $\Lambda_i =
\ZZ e_i \oplus \ZZ e_i'$.

Define now  affine transformations $\gamma_1,  \gamma_2, \gamma_3 \in
\mathbb{A}(3,\CC)$
as follows:

$$
\gamma_1 \begin{pmatrix}
     z_1\\z_2\\z_3
\end{pmatrix} = \begin{pmatrix}
     z_1 + \frac{e_1}{2}\\ - z_2\\ z_3
\end{pmatrix}, \ \ \gamma_2 \begin{pmatrix}
     z_1\\z_2\\z_3
\end{pmatrix} = \begin{pmatrix}
     z_1\\ z_2 + \frac{e_2}{2}\\- z_3
\end{pmatrix}, \ \ \gamma_3 \begin{pmatrix}
     z_1\\z_2\\z_3
\end{pmatrix} = \begin{pmatrix}
     - z_1\\ z_2\\z_3 + \frac{e_3}{2}
\end{pmatrix}, \ \
$$

\noindent
and let $\Gamma \leq \mathbb{A}(3,\CC)$ be the affine group generated
by $\gamma_1, \gamma_2, \gamma_3$
and by the translations by the vectors $e_1, e_1', e_2, e_2', e_3, e_3'$.
\begin{rema}
$\Gamma$ contains the lattice $ \Lambda_1 \oplus \Lambda_2 \oplus \Lambda_3$,
hence $\Gamma$ acts  on $\sE_1 \times \sE_2 \times \sE_3$ inducing a
faithful action
of $G^2:= (\ZZ /2 \ZZ)^3$ on $\sE_1 \times \sE_2 \times \sE_3$.

\end{rema}
We prove next the following
\begin{theo}
Let $S$ be the minimal model of a Burniat surface.
\begin{itemize}
\item[i)] $K^2 = 6$ $\implies$ $\pi_1(S) = \Gamma$, $H_1(S, \ZZ) =
(\ZZ / 2\ZZ)^6$;
\item[ii)] $K^2 = 5$ $\implies$ $\pi_1(S) = \mathbb{H} \oplus (\ZZ /
2 \ZZ)^3$, $H_1(S, \ZZ) = (\ZZ / 2\ZZ)^5$;
\item[iii)] $K^2 = 4$ $\implies$ $\pi_1(S) = \mathbb{H} \oplus (\ZZ /
2 \ZZ)^2$, $H_1(S, \ZZ) = (\ZZ / 2\ZZ)^4$;
\item[iv)] $K^2 = 3$ $\implies$ $\pi_1(S) = \mathbb{H} \oplus (\ZZ /
2 \ZZ), H_1(S, \ZZ) = (\ZZ / 2\ZZ)^3$;
\item[v)] $K^2 = 2$ $\implies$ $\pi_1(S) =  H_1(S, \ZZ) = (\ZZ / 2\ZZ)^3$.
\end{itemize}

Here $\mathbb{H}$ denotes the quaternion group of order $8$.
\end{theo}

\begin{rema}
As already said, these results confirm, except for the case $K^2= 2$,
  the results of
Inoue \cite{inoue}, stating that for $ K^2 \leq 5$ $\pi_1(S) =
\mathbb{H}
\oplus (\ZZ / 2 \ZZ)^{ {K}^2 -2}$.
\end{rema}
\begin{proof}
i) Let $S$ be the minimal model of a Burniat surface with $K_S^2 = 6$.
Then, by the previous section \ref{burniateqinoue},
     $S = X$ has an \'etale
$(\ZZ / 2 \ZZ)^3$ Galois covering $\hat{X}$, which is a hypersurface
of multidegree $(2,2,2)$ in the product of three
elliptic curves $\sE_1 \times \sE_2 \times \sE_3$. Since $\hat{X}$ is
smooth and ample, by Lefschetz's theorem
$\pi_1(\hat{X}) = \pi_1(\sE_1 \times \sE_2 \times \sE_3) \cong
\ZZ^6$.

$\Gamma$ acts on the universal covering of
$\sE_1
\times \sE_2 \times \sE_3 \cong \CC^3$, and acts freely on the
invariant hypersurface $\tilde{X} \subset \CC^3$, the
universal covering of $\hat{X}$, with quotient $S  = X = \tilde{X}/ \Gamma$.
Hence  $\tilde{X}$ is also the universal covering of $S=X$ and
$\pi_1(S) = \Gamma$.

Next we shall prove that $H_1(S, \ZZ) =
(\ZZ / 2\ZZ)^6$.

Since $\gamma_i^2 = e_i$, for $i=1,2,3$, it follows that $\Gamma$ is
generated by $g_1,g_2,g_3, e_1',e_2',e_3'$. It
is clear that
\begin{itemize}
\item[a)] $\gamma_1$ commutes with $e_1, e_1', e_3, e_3'$;
\item[b)] $\gamma_2$ commutes with $e_2, e_2', e_1, e_1'$;
\item[c)] $\gamma_3$ commutes with $e_2, e_2', e_3, e_3'$.
\end{itemize}
Writing   $t_{e_i} \in \mathbb{A}(3,
\CC)$ for the translation by the vector $e_i$,  we see that
$$
\gamma_1 t_{e_2} = t_{e_2}^{-1} \gamma_1 ,\ \ \gamma_1 t_{e_2'} =
t_{e_2'}^{-1} \gamma_1;
$$
$$
\gamma_2 t_{e_3} = t_{e_3}^{-1} \gamma_2, \ \ \gamma_2 t_{e_3'} =
t_{e_3'}^{-1} \gamma_2;
$$
$$
\gamma_3 t_{e_1} = t_{e_1}^{-1} \gamma_3, \ \ \gamma_3 t_{e_1'} =
t_{e_1'}^{-1} \gamma_3.
$$

This implies that $2e_1, 2e_1',2e_2,2e_2',2e_3,2e_3' \in [\Gamma,
\Gamma]$.
Moreover,
$$
\gamma_1\gamma_2
\begin{pmatrix}
z_1 \\
z_2\\
z_3
\end{pmatrix} =
\begin{pmatrix}
z_1 + \frac{e_1}{2} \\
-z_2 - \frac{e_2}{2}\\
-z_3
\end{pmatrix}
= t_{e_2}^{-1}\gamma_2\gamma_1,
$$
whence $e_2 \in [\Gamma, \Gamma]$. Similarly, we see that (as the respective
commutators of
    $\gamma_1$ with $ \gamma_3$, $\gamma_2$ with $\gamma_3$) $e_1,e_3 
\in [\Gamma,
\Gamma]$.

Therefore $\Gamma' :=
\Gamma / \langle e_1,e_2,e_3,2e_1',2e_2', 2e_3' \rangle$ surjects
onto $\Gamma^{ab}$.

But $\Gamma'$ is already
abelian, since the morphism
$$
\Gamma' \ra (\ZZ /2 \ZZ)^3 \oplus (\ZZ /2 \ZZ)^3,
$$
    mapping the residue classes of $\gamma_1, \gamma_2, \gamma_3,
e_1',e_2',e_3'$ onto the ordered set of coordinate vectors
of $(\ZZ /2 \ZZ)^3$ is easily seen to be well defined and an
isomorphism. This shows that $H_1(S, \ZZ) =
(\ZZ / 2\ZZ)^6$.

In order to prove the assertions $ii) - v)$
observe preliminarly that, if $X$ is the above singular model of $S$, 
then, by van
Kampen's theorem, $\pi_1(S) \cong \pi_1(X)$.
Therefore, for the remaining cases, it suffices to calculate $\pi_1(X)$.

Let $X$ be the above singular model of a Burniat surface with $K^2 \leq 5$.
Consider the
$G^2 \cong (\ZZ / 2 \ZZ)^3$-Galois cover $\hat{X}$.

Since the singularities of $\hat{X}$ are only nodes,
$\pi_1(\hat{X})  \cong \ZZ^6$ by the theorem of
Brieskorn-Tyurina (cf. \cite{brieskorn1}, \cite{brieskorn2}, \cite{tjurina}).

By \cite{armstrong1}, \cite{armstrong2} $\pi_1(X) \cong \Gamma / 
\Tors(\Gamma)$, where
$\Tors(\Gamma)$ is the normal
subgroup of $\Gamma$ generated by all elements of $\Gamma$ having
fixed points on the universal covering $\tilde{X}$ of
$\hat{X}$ (which is, as we have seen before, a $\Gamma$-invariant
hpersurface in $\CC^3$).

Note that the elements in $G^2 \cong (\ZZ / 2 \ZZ)^3$ induced by the elements
$$
\gamma_1,\gamma_2,\gamma_3,\gamma_1\gamma_2,\gamma_1\gamma_3,
\gamma_2\gamma_3 \in \Gamma
$$
do not have fixed points on $\sE_1 \times \sE_2 \times \sE_3$. Instead,
$$
\gamma_1 \gamma_2 \gamma_3\begin{pmatrix}
z_1 \\
z_2\\
z_3
\end{pmatrix} =
\begin{pmatrix}
-z_1 + \frac{e_1}{2} \\
-z_2 - \frac{e_2}{2}\\
-z_3-\frac{e_3}{2}
\end{pmatrix}
$$
has as fixed points the $64$ points on $\sE_1 \times \sE_2 \times
\sE_3$  corresponding to vectors in $\CC^3$ such that
\begin{equation}\label{fp}
2z_i \equiv  \frac{e_i}{2} \mod \Lambda_i , \forall i.
\end{equation}
Equivalently,
\begin{equation}\label{fp2}
z_i \equiv  \frac{e_i}{4} \mod \frac12 \Lambda_i , \forall i.
\end{equation}
ii) Let $X$ be the singular model of a Burniat surface with $K^2 =
5$. Then $\hat{X}$ has $4$ nodes (lying over the point
$P_4
\in
\PP^2$, see figure \ref{configs}).

We observed that if  $\gamma \in \Gamma$
has a fixed point on $\tilde{X}$, then there is a
$$\lambda \in \ZZ^6 \cong \langle e_1,e_2,e_3,e_1',e_2',e_3' \rangle
=: \Lambda$$
  such that
$$\gamma = \gamma_1
\gamma_2 \gamma_3 t_{\lambda}.
$$

  Let now $z =(z_1,z_2,z_3) \in \tilde{X}
\subset \CC^3$. Then $z$ yields a fixed point of
$\gamma_1 \gamma_2 \gamma_3$ on $\hat{X}$ if and only if there is a
$\hat{\lambda} \in \Lambda$  such that
$$
2\begin{pmatrix}
z_1 \\
z_2\\
z_3
\end{pmatrix}
=
\frac{1}{2}\begin{pmatrix}
e_1 \\
-e_2\\
-e_3
\end{pmatrix} + \hat{\lambda} \iff z = \frac{1}{4} \epsilon +
\frac{\hat{\lambda}}{2},
$$
where we have set $\epsilon := \begin{pmatrix}
e_1 \\
-e_2\\
-e_3
\end{pmatrix}$.

We show now that $z$ is a fixed point of a $\ga$ as above iff 
$\lambda = - \hat{\lambda}$.

In fact:

$$
  \gamma(z) = \gamma_1 \gamma_2 \gamma_3 t_{\lambda} (z) = -(z+\lambda)
+ \frac{1}{2} \epsilon =
$$
$$
-\frac{1}{4} \epsilon - \frac{\hat{\lambda}}{2} - \lambda +
\frac{1}{2} \epsilon = z - \lambda - \hat{\lambda}.
$$

Modifying $z$ modulo
$\Lambda$, we replace $z$ by $ z + \la'$, and the corresponding
$\hat{\lambda}$ gets replaced by $\hat{\lambda} + 2 \la'$; hence we see that
$\gamma_1 \gamma_2 \gamma_3 t_{\lambda}$ has a fixed point on $\tilde{X}$ for
all $\lambda \in -\hat{\lambda} +
2\Lambda$.

Therefore $2 \Lambda$ is contained in
$\Tors(\Gamma)$. Since the above arguments
apply to all remaining cases (ii) - (v) we summarize what we have
seen in  the following
\begin{lemma}\label{exseq}
If $\Gamma$ has a fixed point $z$ on the universal covering of
$\hat{X}$ (i.e., we are in one of the cases ii) - v)), then
$\pi_1(X)$ is a quotient  of $\bar{\Gamma} := \Gamma /
2 \Lambda$.

We have thus an exact sequence
$$
1 \ra (\ZZ/2\ZZ)^6 \ra \bar{\Gamma} \ra (\ZZ/2\ZZ)^3 \ra 1.
$$

\end{lemma}
In particular, we already showed that the fundamental group of a 
Burniat surface with $K^2
\leq 5$ is finite: we are now going to write its structure explicitly.

\begin{rema}\label{wichtig}
1) The images of $e_i, e_j'$, $i,j \in \{1,2,3 \}$  in
$\bar{\Gamma}$ are contained in the center of
$\bar{\Gamma}$, i.e., the above exact sequence yields a central extension.

\noindent
2) Note that over each $(1,1,1)$ point of the branch divisor $D \subset
\PP^2$ there are $4$ nodes of $\hat{X}$, which are
a $G^2$-orbit of fixed points of $\gamma_1 \gamma_2\gamma_3$ on
$\hat{X}$.
Let $z \in \tilde{X}$ induce  a fixed point of $\gamma_1 \gamma_2\gamma_3$ on
$\hat{X}$: then $z = \frac{1}{4} \epsilon +
\frac{\hat{\lambda}_z}{2}$, and the other 3 fixed points in the orbit 
are exactly
the points induced by  $\gamma_i(z)$
on $\hat{X}$, for $ i=1,2,3$. We have:

$$
\gamma_1(z) = \gamma_1 \begin{pmatrix}
                         \frac{e_1}{4} + \frac{1}{2} (\hat{\lambda}_z)_1 \\
\frac{-e_2}{4} - \frac{1}{2} (\hat{\lambda}_z)_2 \\
\frac{-e_3}{4} + \frac{1}{2} (\hat{\lambda}_z)_3
                        \end{pmatrix} = \frac{1}{4} \epsilon + 
\frac{\hat{\lambda}_{\gamma_1(z)}}{2}.
$$

This implies that
$$
\hat{\lambda}_{\gamma_1(z)} \equiv \hat{\lambda}_z + \begin{pmatrix}
                                                       e_1 \\e_2\\ 0
                                                      \end{pmatrix} 
\mod 2\Lambda,
$$
and similarly
$$
\hat{\lambda}_{\gamma_2(z)} \equiv \hat{\lambda}_z + \begin{pmatrix}
                                                       0 \\e_2\\ e_3
                                                      \end{pmatrix} 
\mod 2\Lambda, \
\hat{\lambda}_{\gamma_3(z)} \equiv \hat{\lambda}_z + \begin{pmatrix}
                                                       e_1 \\0\\ e_3
                                                      \end{pmatrix} 
\mod 2\Lambda.
$$

\noindent
3) Let $X$ be the singular model of a Burniat surface, and  choose w.l.o.g.
one of the points in $\tilde{X}$ over the $(1,1,1)$ - point $P_4$ to 
be $z:=\frac{1}{4} \epsilon$.
This is equivalent to
$\hat{\lambda}_z = 0$. For each $(1,1,1)$ point of the branch
divisor $D \subset \PP^1$ choose one singular point of $\hat{X}$
lying over it.

Let $\mathcal{S} := \{z(4) = z,
\ldots, z(9-K^2) \}$ be a choice of representatives for
each $G^2$-orbit of points of $\hat{X}$ lying
over the respective $(1,1,1)$ - points. Then:
$$
\pi_1(X) = \Gamma / \langle \gamma_1
\gamma_2
\gamma_3 t_{\lambda} : \lambda \in -\hat{\lambda}_z + 2\Lambda, \ z
\in \mathcal{S} \rangle .
$$
In particular, we have the relations: $$\gamma_1 \gamma_2 \gamma_3 = 
1,$$ and, by 2) :
$$e_1 = e_2 = e_3.$$
\end{rema}

Recall  now that $\gamma_i^2 = e_i$. Therefore  in $\pi_1(X)$ we 
have: $$\gamma_1^2 =
\gamma_2^2 = \gamma_3^2 = e_1 + e_2 + e_3.$$

Thus we get an exact sequence (cf. lemma
\ref{exseq}):
$$
1 \ra (\ZZ/2\ZZ)^3 \oplus (\ZZ/2\ZZ) \ra \pi_1(X) \ra (\ZZ/2\ZZ)^2 \ra 1,
$$

where the map $\varphi:\pi_1(X) \ra (\ZZ/2\ZZ)^2$ is given by 
$\gamma_1 \mapsto (1,0)$,
$\gamma_2 \mapsto (0,1)$, $e_i' \mapsto 0$.
This immediately shows that the kernel of $\varphi$ is equal to
$\langle e_1', e_2' e_3', e_1 + e_2 + e_3 = \gamma_i^2 \rangle$.

Let $\mathbb{H} := \{ \pm 1, \pm i, \pm j, \pm k \}$ be the
quaternion group, and let $\ga_1, \ga_2, \ga_3$ correspond
respectively to $i, j, -k$: then  we obtain an isomorphism
$$
\pi_1(X) \cong
\mathbb{H} \oplus
(\ZZ / 2 \ZZ)^3.
$$
This
proves the assertion on the fundamental group for Burniat surfaces
with $K^2 = 5$. That $H_1(S, \ZZ) = (\ZZ / 2
\ZZ)^5$ follows, since $\mathbb{H}^{ab} = (\ZZ / 2
\ZZ)^2$.

iii), iv) First observe that, by the above, if $X$ is the singular
model of a Burniat surfaces with $K^2 \leq 5$, then
$\pi_1(X)$ is the quotient of $\mathbb{H} \oplus
(\ZZ / 2 \ZZ)^3$ by the relations $\hat{\lambda}_{z(i)} 
-\hat{\lambda}_{z(j)} = 0$, where
$z(i)\neq z(j) \in \mathcal{S}$.

Note that for
$K^2 =4$ (nodal or non nodal), the projections of $z(4)$ and $z(5)$ 
to  $\sE_2$, resp. $\sE_3$,
are points whose differences  are
non trivial $2$-torsion elements. Since each of the corresponding
$(1,1,1)$ points lies on two different lines $D_{2,2}, \
D_{2,3}$, respectively $D_{3,2}, \
D_{3,3}$,
hence the images of $z(4)$ and $z(5)$ under the composition
of the projection
to  $\sE_2$ (resp. to $\sE_3$) with the quotient map $\sE_2
\ra \PP^1$ (resp. $\sE_3 \ra \PP^1$) have different
$x_2$-value (resp. $x_3$-value).

\noindent
{\bf Claim.} The image of $\hat{\lambda}_{z(4)} -\hat{\lambda}_{z(5)}$
in $\oplus_{i=1}^3 e_i'\ZZ / 2 \ZZ$ is non zero.

{\em Proof of the claim}.

Again, we look at the image of $z(4)$ (resp.$z(5)$)  in $\sE_2 \ra
\PP^1$ (with coordinate of $\PP^1$ equal to $x_2$).
We have seen that the corresponding
$(1,1,1)$ points
$P_4$,
$P_5$ lie on two different lines $D_{2,2}, \
D_{2,3}$, respectively $D_{3,2}, \
D_{3,3}$; hence   the  respective $x_2$ coordinates of the projection of
$z(4)$ and $z(5)$ to $\PP^1$ are different.

We conclude since  the description of the transformations of order $2$
of $\sE_2$ given by translation by 2-torsion elements (cf.section 
\ref{legendre}) shows that
translation by $\frac{e_2}{2}$ is the only one which leaves the $x_2$ 
coordinate invariant.

\hfill{\em QED for the claim}.

\smallskip

Therefore, if $X$ is the singular model of a Burniat surface with $K^2 =4$,
$\pi_1(X)$ is the quotient of  $\mathbb{H} \oplus
(\ZZ / 2 \ZZ)^3$ by an element having a non trivial component in 
$(\ZZ / 2 \ZZ)^3$,
hence $\pi_1(X) \cong \mathbb{H} \oplus
(\ZZ / 2 \ZZ)^2$.

Assume now that  $X$ is the singular model of a Burniat surface
with $K^2 =3$. Here the branch divisor on $\PP^2$
has three $(1,1,1)$-points. Repeating the above argument, we see that
  $\pi_1(X)$ is the quotient of $\mathbb{H} \oplus
(\ZZ / 2 \ZZ)^3$ by $\hat{\lambda}_{z(4)} -\hat{\lambda}_{z(5)}$ and
$\hat{\lambda}_{z(4)} -\hat{\lambda}_{z(6)}$.

As
above, we look at the image in $\oplus_{i=1}^3 e_i'\ZZ / 2 \ZZ$ and see
that they give (up to a permutation of indices) the
elements $\begin{pmatrix}
0\\
1\\
1
\end{pmatrix}$, $\begin{pmatrix}
1\\
0\\
1
\end{pmatrix}$. This implies that $\pi_1(X)$ is the quotient of
$\mathbb{H} \oplus
(\ZZ / 2 \ZZ)^3$ by two linear independent relations in $(\ZZ / 2
\ZZ)^3$. Therefore $\pi_1(X) = \mathbb{H} \oplus
(\ZZ / 2 \ZZ)$.

v) Let $X$ be the singular model of a Burniat surfaces with $K^2 =2$.
\begin{rema}
Observe that, by \cite{miles}, \cite{milesLNM}, \cite{miyaoka},
$|\pi_1(X)| \leq 9$. Since $\pi_1(X)$ is a quotient of $\mathbb{H} \oplus
(\ZZ / 2 \ZZ)$ by a relation coming from an element of $\Lambda$ 
there are only two possibilities: either
$\pi_1(X) =
\mathbb{H}$ or $\pi_1(X) = (\ZZ / 2 \ZZ)^3$.
\end{rema}

We are going to show that the second alternative holds.

Here the branch divisor $D$ on $\PP^2$ has four
$(1,1,1)$-points $P_4, P_5,P_6,P_7$. As above, $\pi_1(X)$ is the 
quotient of $\mathbb{H} \oplus
(\ZZ / 2 \ZZ)^3$ by the relations:
$$\hat{\lambda}_{z(4)} -\hat{\lambda}_{z(5)} = 0, \
\hat{\lambda}_{z(4)} -\hat{\lambda}_{z(6)}=0, \ \hat{\lambda}_{z(4)} 
-\hat{\lambda}_{z(7)}=0,
$$
where $z(j) \in \hat{X}$ is a point lying over $P_{j}$.

Looking at figure \ref{configs} of the configuration of lines in 
$\PP^2$ for the Burniat surface with
$K^2 = 2$, we see that
$$
P_4, P_5 \in D_{1,2}, \ P_4, P_6 \in D_{2,2}, \ P_4, P_7 \in D_{3,2},
$$
i.e., $P_4, P_5$ lie in the same green line, but in two different red 
and two different black lines,
$P_4, P_6$ lie in the same red line, but in two different green and 
two different black lines,
and $P_4, P_7$ lie in
the same black line, but in two different green and two different red lines.

This means that if we look at the image of $\hat{\lambda}_{z(4)} 
-\hat{\lambda}_{z(5)}$,
$\hat{\lambda}_{z(4)} -\hat{\lambda}_{z(6)}$,  $\hat{\lambda}_{z(4)} 
-\hat{\lambda}_{z(7)}$ in
  $\oplus_{i=1}^3 e_i'\ZZ / 2 \ZZ$ we see
that they give (up to a permutation of indices) the
elements
$$
\begin{pmatrix}
0\\
1\\
1
\end{pmatrix},  \ \begin{pmatrix}
1\\
0\\
1
\end{pmatrix}, \ \begin{pmatrix}
1\\
1\\
0
\end{pmatrix}.
$$
These three vectors are linearly dependent, hence, taking the 
quotient by these relations, the
  rank of $\oplus_{i=1}^3 e_i'\ZZ / 2 \ZZ$ drops only by two.

In order to determine the component of
the image of $\hat{\lambda}_{z(4)} -\hat{\lambda}_{z(5)}$,
$\hat{\lambda}_{z(4)} -\hat{\lambda}_{z(6)}$,  $\hat{\lambda}_{z(4)} 
-\hat{\lambda}_{z(7)}$ in the center of
the quaternion group, we have to write the points $z(j)$, $j \in 
\{4,5,6,7 \}$ more explicitly, using section
\ref{legendre}.

Observe that in the case $K^2 = 2$, we have $\sE_1 = \sE_2 = \sE_3 =: \sE$.

The fixed points of $\gamma_1
\gamma_2 \gamma_3$ are given by $w_i = 0, \ i= 1,2,3$. Setting $x_i' 
= 1$, and $a:=\sqrt{b}$, we can assume
w.l.o.g. that
$$
z(4) =  \begin{pmatrix}
(1:b) , (1:a), (1:0) \\
(1:b) , (1:a), (1:0) \\
(1:b) , (1:a), (1:0)
\end{pmatrix}.
$$

By equation (\ref{constant}), we have $v_1 v_2 v_3 = c v'_1 v'_2 
v'_3$, whence $c=a^3$.

W.l.o.g., by (2) of remark \ref{wichtig}, we can assume that $z(5)$ is given by
$$
z(5) =  \begin{pmatrix}
(1:b) , (1: \zeta), (1:0) \\
(b:1) , (a:1), (1:0) \\
(b:1) , (a:1), (1:0)
\end{pmatrix} =
\begin{pmatrix}
(1:b) , (1: \zeta), (1:0) \\
f_2((1:b) , (1:a), (1:0) )\\
f_2((1:b) , (1:a), (1:0))
\end{pmatrix}.
$$
We have now to determine $\zeta$ in such a way that $v_1 v_2 v_3 = 
a^3v'_1 v'_2 v'_3$.

For $z(5)$ we have
\begin{equation}\label{aa}
v_1 v_2 v_3 = \zeta = c v'_1 v'_2 v'_3 = a^3 a^2 = a^5.
\end{equation}
Since by section \ref{legendre} the only two translations of order 
$2$ leaving $(x':x)$ unchanged are the identity
and $g_1$, we have
$$
((1:b) , (1: \zeta), (1:0)) = ((1:b) , (1:a), (1:0))
$$ or
$$
((1:b) , (1: \zeta), (1:0)) = g_1((1:b) , (1:a), (1:0)) = ((1:b) , 
(1:-a), (1:0)).
$$
Together with equation (\ref{aa}) we get: $a^5 = \zeta = \pm a$, 
i.e., $a^4 = \pm 1$. $a^4 = 1$ is not possible,
because this would imply that $b =a^2 = \pm1$, a contradiction to $ b 
\neq \frac1b$.

Hence we have that $a^4 =
-1$, i.e.,
$\zeta = a^5 = -a$, and we see that the the relation 
$\hat{\lambda}_{z(4)} -\hat{\lambda}_{z(5)} = 0$
is given by
$$
\langle \begin{pmatrix}
0\\
1\\
1
\end{pmatrix}, 1 \rangle \in \bigoplus_{i=1}^3 e'_i(\ZZ /2 \ZZ) 
\oplus \ZZ/2 \ZZ,
$$

where the last summand is the center of the quaternion group.

Using  for $z(6)$ and $z(7)$ the same argument as for $z(5)$, we get 
two further elements which have to be
set equal to zero in the quotient:
$$
\langle \begin{pmatrix}
1\\
0\\
1
\end{pmatrix}, 1 \rangle, \ \langle \begin{pmatrix}
1\\
1\\
0
\end{pmatrix}, 1 \rangle \in \bigoplus_{i=1}^3 e'_i(\ZZ /2 \ZZ) 
\oplus \ZZ/2 \ZZ.
$$
Taking the sum of these three elements in $\bigoplus_{i=1}^3 e'_i(\ZZ 
/2 \ZZ) \oplus \ZZ/2 \ZZ$ we see that we
get
$$
\langle \begin{pmatrix}
0\\
0\\
0
\end{pmatrix}, 1 \rangle  = 0,
$$
and we have concluded the proof of the theorem.

\end{proof}

\begin{rema}
We shall show in the last section how Burniat surfaces with $K^2 = 2$ are
{\em classical} Campedelli surfaces, i.e.,
obtained as the tautological $(\ZZ /2 \ZZ)^3$ Galois covering of 
$\PP^2$ branched on seven lines.
\end{rema}

\section{The moduli space of primary Burniat surfaces }\label{moduli}

    In  this section we finally devote ourselves to the main result of 
the paper.
First of all, we show

\begin{theo}\label{locmod}
The subset  of the Gieseker moduli space corresponding to
primary Burniat surfaces is an irreducible connected component,
normal, unirational and   of dimension
equal to 4.
\end{theo}

This result was already proven in \cite{mlp} using the fact that the
bicanonical map of the canonical model $X'$ of a Burniat surface
is exactly the bidouble covering $ X' \ra Y'$ onto  the normal Del Pezzo
surface $Y'$ of degree $K^2_{X'}$
obtained as the anticanonical model of
the weak Del Pezzo surface obtained blowing up not only the points 
$P_1, P_2, P_3$, but also all
the other triple points of $D$.

We shall now give an alternative proof of their theorem.

\medskip

\begin{proof}
The singular model $X$ of a  primary Burniat surface is smooth,
and has ample canonical divisor.
Hence it equals the minimal model $S $ (and the canonical model $X'$).

  Since $ \hat{X} \ra X$ is \'etale with
group $G^2$, it suffices to show that
the Kuranishi family of $\hat{X}$ is smooth. Then it will also
follow that the Kuranishi family of $X$ is smooth, whence the Gieseker
moduli space is normal (being locally analytically isomorphic to the
quotient of the base of the Kuranishi family by the finite group
$ \Aut(X)$).

Since $\hat{X} \subset \sE_1 \times \sE_2 \times \sE_3$ is a smooth
hypersurface, setting for convenience
$ T : = \sE_1 \times \sE_2 \times \sE_3$, we have the tangent bundle sequence
$$ 0 \ra  \Theta_{ \hat{X}}  \ra \Theta_T \otimes \hol_{ \hat{X}}
\cong \hol_{ \hat{X}}^3 \ra
\hol_{ \hat{X}}( \hat{X}) \ra 0 $$
with exact cohomology sequence

$$0 \ra  \CC^3 \ra H^0 (\hol_{ \hat{X}}( \hat{X}))\cong \CC^{10} \ra $$
$$ \ra H^1 (\Theta_{ \hat{X}}) \ra  H^1 (\Theta_T
\otimes
\hol_{
\hat{X}}) \cong \CC^9
\ra   H^1 (\hol_{ \hat{X}}(
\hat{X})) \cong \CC^3.
$$

Since  $\hat{X} $ moves in a smooth family of dimension $ 13= 6 + 7$,
a fibre bundle over the family
of  deformations of the
principally  polarized Abelian variety $T$, with fibre the linear
system  $\PP( H^0 ( T, \hol_T (\hat{X}))) $,
it suffices to show that the map $H^1 (\Theta_T \otimes
\hol_{\hat{X}})\ra   H^1 (\hol_{ \hat{X}}(\hat{X}))$ is surjective.

It suffices to observe that $H^1 (\Theta_T \otimes
\hol_{\hat{X}}) \cong H^1 (\Theta_T) $, $H^1 (\hol_{
\hat{X}}(\hat{X})) \cong H^2 (\hol_T)$,
and that, as well known, the above map corresponds via these isomorphisms
to the contraction with
the first Chern class of $\hat{X}$, an element of $ H^1 (\Omega^1_T)$
which represents a non degenerate
alternating form. Whence surjectivity follows.

Thus the base of the Kuranishi family of $\hat{X}$ is smooth 
(moreover the Kodaira Spencer
map of the above family
is a bijection, but we omit the verification here), whence the base 
of the Kuranishi
family of $X$, which is the $G^2$-invariant
part of the base of the Kuranishi family  of $\hat{X}$, is also smooth.

Moreover  the Kuranishi family of $X$ fibres onto the family of
$G^2$-invariant  deformations of $T$, which coincides
    with the
deformations of the three individual elliptic curves.

  The fibres of the corresponding morphism between the bases
of the respective families are
given by the $G^2$-invariant part of the
linear system $|\hat{X} |$, which we are going to calculate
explicitly as being isomorphic to $\PP^1$.

We obtain thereby  a rational family of dimension 4  parametrizing the
primary Burniat
surfaces. This proves the  unirationality of the  4 dimensional
irreducible component.

That the irreducible component of the moduli space is in fact a
connected component follows from the
more general result below (theorem \ref{homotopy}).

We  calculate now
$$
H^0(\sE_1 \times \sE_2 \times \sE_3, p_1^*\hol_{\sE_1}([o_1] +
[\frac{e_1}{2}]) \otimes p_2^*\hol_{\sE_2}
([o_2] + [\frac{e_2}{2}])\otimes p_3^*\hol_{\sE_3}([o_3] +
[\frac{e_3}{2}]))^{G^2},
$$
where $p_i: \sE_1 \times \sE_2 \times \sE_3 \rightarrow \sE_i$ is the
$i$ - th projection.

   From lemma \ref{h++} it follows that

\begin{multline*}
    H_1:= H^0(\sE_1, \hol_{\sE_1}([o_1] + [\frac{e_1}{2}])) = H_1^{+++}
\oplus H_1^{-+-}=\\
= H^0(\sE_1, \hol_{\sE_1}([o_1] + [\frac{e_1}{2}]))^{+++} \oplus
H^0(\sE_1, \hol_{\sE_1}([o_1] + [\frac{e_1}{2}]))^{-+-},
\end{multline*}

\begin{multline*}
    H_2:= H^0(\sE_2, \hol_{\sE_2}([o_2] + [\frac{e_2}{2}])) = H_2^{+++}
\oplus H_2^{--+}=\\
= H^0(\sE_2, \hol_{\sE_2}([o_2] + [\frac{e_2}{2}]))^{+++} \oplus
H^0(\sE_2, \hol_{\sE_2}([o_2] + [\frac{e_2}{2}]))^{--+},
\end{multline*}

\begin{multline*}
    H_3:= H^0(\sE_3, \hol_{\sE_3}([o_3] + [\frac{e_3}{2}])) = H_3^{+++}
\oplus H_3^{+--}\\
= H^0(\sE_3, \hol_{\sE_3}([o_3] + [\frac{e_3}{2}]))^{+++} \oplus
H^0(\sE_3, \hol_{\sE_3}([o_3] + [\frac{e_3}{2}]))^{+--}.
\end{multline*}
As a consequence of this, we get

\begin{multline*}
H^0(\sE_1 \times \sE_2 \times \sE_3, p_1^*\hol_{\sE_1}([o_1] + [\frac{e_1}{2}])
\otimes p_2^*\hol_{\sE_2}([o_2] + [\frac{e_2}{2}])\otimes
p_3^*\hol_{\sE_3}([o_3] + [\frac{e_3}{2}]))^{G^2} = \\
(H_1^{+++} \oplus H_2^{+++} \oplus H_3^{+++}) \oplus (H_1^{-+-}
\oplus H_2^{--+} \oplus H_3^{+--}) \cong \CC^2.
    \end{multline*}

\end{proof}

We have obtained a 4 - dimensional rational family parametrizing  all the
primary Burniat surfaces.

This can also be  seen in a more direct fashion by the fact that,
  fixing 4 points in $\PP^2$ in general position,  we
can fix the 3 lines $ D_{i,1}, i=1,2,3$ and 2 lines $ D_{1,2},  D_{2,2}$.
Then the other 4 lines vary each in a pencil, hence we
get 4 moduli.

    In the remaining part of this section, we will prove the following result:

\begin{theo}\label{homotopy}
    Let $S$ be a smooth complex projective surface which is
homotopically equivalent to a primary Burniat surface.
Then $S$ is a Burniat surface.

\end{theo}

\begin{proof}

\noindent
Let $S$ be a smooth complex projective surface with $\pi_1(S) =
\Gamma$.

Recall that $\gamma_i^2 = e_i$ for $i = 1,2,3$. Therefore
$\Gamma = \langle \gamma_1, e_1', \gamma_2, e_2', \gamma_3, e_3'
\rangle$ and we have the exact sequence
$$
1 \rightarrow \ZZ^6 \cong \langle e_1, e_1', e_2, e_2',e_3,e_3'
\rangle \rightarrow \Gamma \rightarrow (\ZZ / 2 \ZZ)^3 \rightarrow 1,
$$

where $e_i \mapsto \gamma_i^2$.

If we set $\Lambda_i:= \ZZ e_i \oplus \ZZ e_i'$, $i=1,2,3$ then
$$\pi_1 (\sE_1 \times \sE_2 \times \sE_3) = \Lambda_1
\oplus
\Lambda_2\oplus \Lambda_3.$$

We also have the three lattices $\Lambda_i' := \ZZ \frac{e_i}{2}
\oplus \ZZ e_i'$.

\begin{rema}
    1) $\Gamma$ is a group of affine transformations on $\Lambda_1'
\oplus \Lambda_2' \oplus \Lambda_3'$.

\noindent
2) We have an \'etale double cover $\sE_i = \CC / \Lambda_i
\rightarrow \sE_i' := \CC / \Lambda_i'$, which is the
quotient by the semiperiod $\frac{e_i}{2}$ of $\sE_i$.
\end{rema}

\noindent
$\Gamma$ has the following three subgroups of index two:
$$
\Gamma_3 := \langle \gamma_1, e_1', e_2, e_2',\gamma_3, e_3' \rangle,
$$
$$
\Gamma_1 := \langle \gamma_1, e_1', \gamma_2, e_2',e_3, e_3' \rangle,
$$
$$
\Gamma_2 := \langle e_1, e_1', \gamma_2, e_2',\gamma_3, e_3' \rangle,
$$

\noindent
corresponding to three \'etale double covers of $S$: $S_i \rightarrow
S$, for $i = 1,2,3$.

\begin{lemma}
The Albanese variety of $S_i$ is $\sE'_i$.

In particular, $q(S_1) =
q(S_2) = q(S_3) = 1$.
\end{lemma}

\begin{proof}
Observe once more that
\begin{itemize}
    \item[i)] $\gamma_1$ commutes with $e_1,e_1',e_3,e_3'$;
\item[ii)] $\gamma_2$ commutes with $e_2,e_2',e_1,e_1'$;
\item[iii)] $\gamma_3$ commutes with $e_2,e_2',e_3,e_3'$.
\end{itemize}

Denoting by $t_{e_i}\in \mathbb{A}(3, \CC)$ the translation with
vector  $e_i$,   we see that
$$
\gamma_1 t_{e_2} = t_{e_2}^{-1} \gamma_1 \ \ \gamma_1 t_{e_2'} =
t_{e_2'}^{-1} \gamma_1;
$$
$$
\gamma_3 t_{e_1} = t_{e_1}^{-1} \gamma_3 \ \ \gamma_3 t_{e_1'} =
t_{e_1'}^{-1} \gamma_3.
$$

This implies that $2e_2,2e_2',2e_1,2e_1' \in [\Gamma_3, \Gamma_3]$.

Moreover, $\gamma_1 \gamma_3 = t_{e_1}^{-1}\gamma_3 \gamma_1$, whence
already $e_1 \in [\Gamma_3,
\Gamma_3]$.

Therefore we have a surjective homomorphism
$$
\Gamma_3' := \Gamma_3 / \langle 2e_2, 2e_2',e_1, 2e_1' \rangle =
\Gamma_3 / (2 \ZZ^3 \oplus \ZZ)  \rightarrow
\Gamma_3^{ab} = \Gamma_3 / [\Gamma_3, \Gamma_3].
$$
Since the images of $\gamma_3$ and $e_3'$ are in the centre of
$\Gamma_3'$, we get
that  $\Gamma_3'$ is abelian, hence $ H_1(S_3, \ZZ ) = \Gamma_3^{ab}
= \Gamma_3' $ and
$$
\Gamma_3' = \langle \gamma_3, e_3' \rangle \oplus (\ZZ / 2 \ZZ )^4 =
\ZZ \frac{e_3}{2}
\oplus \ZZ e_3'\oplus (\ZZ / 2
\ZZ )^4 = \Lambda_3 \oplus (\ZZ / 2 \ZZ )^4.
$$

This implies that
$\Alb(S_3) = \CC / \Lambda_3' = \sE'_3$.

\noindent
The same calculation shows that $\Gamma_i^{ab} = H_1(S_i, \ZZ ) =
\Lambda_i' \oplus (\ZZ / 2 \ZZ)^4$, whence
$\Alb(S_i) = \CC / \Lambda_i' = \sE_i'$, also for $i= 2,3$.

\end{proof}

Let now $\hat{S} \rightarrow S$ be the \'etale $(\ZZ / 2 \ZZ)^3$ -
covering associated to $\ZZ^6 \cong \langle e_1, e_1',
e_2, e_2', e_3, e_3'\rangle \triangleleft \Gamma$. Since $\hat{S}
\rightarrow S_i \rightarrow S$, and $S_i$ maps to
$\sE'_i$ (via the Albanese map), we get a morphism
$$
f:\hat{S} \rightarrow \sE'_1 \times \sE'_2 \times \sE'_3 = \CC/
\Lambda_1' \times \CC / \Lambda_2' \times \CC /
\Lambda_3'.
$$
Since the covering of $\sE'_1 \times \sE'_2 \times \sE'_3$ associated to
$\Lambda_1 \oplus \Lambda_2 \oplus \Lambda_3 <  \Lambda_1' \oplus
\Lambda_2' \oplus \Lambda_3'$ is
$\sE_1 \times \sE_2 \times \sE_3$,  we see that $f$ factors through
$\sE_1 \times \sE_2 \times \sE_3$ and the
Albanese map of
$\hat{S}$ is $\hat{\alpha} : \hat{S} \rightarrow \sE_1 \times \sE_2
\times \sE_3$.

\medskip
\noindent
Let $Y:= \hat{\alpha}(\hat{S}) \subset T =\sE_1 \times \sE_2 \times
\sE_3$ be the Albanese image of $\hat{S}$.

\noindent
We consider for $i \neq j \in \{1,2,3\}$, the natural projections
$$
\pi_{ij} : \sE_1 \times \sE_2 \times \sE_3 \rightarrow \sE_i \times \sE_j.
$$

\begin{claim}
If $S$  is homotopically equivalent to a primary Burniat surface,
then for $i \neq j \in \{1,2,3\}$ we have
$\deg (\pi_{ij} \circ \hat{\alpha}) = 2$.
\end{claim}

\begin{proof}
The degree of $\pi_{ij} \circ \hat{\alpha}$ is the index of the
image of $H^4 ( \sE_i \times \sE_j , \ZZ)$
inside   $H^4 ( \hat{S} , \ZZ)$. But the former equals $\wedge ^4 (
\Lambda_i \oplus \Lambda_j)$,
hence we see that this number is an invariant of the cohomology
algebra of $\hat{S}$.

\end{proof}

The above claim implies that $\hat{S} \rightarrow Y$ is a birational
morphism and that $Y \subset Z$ has multidegree
$(2,2,2)$. Thus
$K_{Y} = \hol_{Y}(Y) = \hol_Y(2,2,2)$, and $K_Y^2 = (2,2,2)^3 = 48$.
On the other hand, since $S$ is homotopically
equivalent to a primary Burniat surface, we have that $K_S^2 = 6$,
whence $K_{\hat{S}}^2 = 6 \cdot 2^3 = 48$.

Moreover, we have
$$
p_g(\hat{S}) = q(\hat{S}) + \chi(\hat{S}) -1 = 3 + 8\chi(S) -1 = 10.
$$
The short exact sequence

$$
0 \rightarrow \hol_T \rightarrow \hol_T(Y) \rightarrow \omega_Y \rightarrow 0,
$$

induces a long exact cohomology sequence
\begin{multline*}
    0 \rightarrow H^0(T,\hol_T) \rightarrow H^0(T,\hol_T(Y)) \rightarrow
H^0(Y, \omega_Y) \rightarrow \\
\rightarrow H^1(T,\hol_T) \rightarrow H^1(T,\hol_T(Y)) = 0,
\end{multline*}
where the last equality holds since $Y$ is an ample divisor on $T$.

Moreover $H^0(T,\hol_T) \cong \CC$, and
$H^1(T,\hol_T) \cong \CC^3$, and therefore
$$
p_g(Y) = h^0(Y, \omega_Y) = 10 = p_g(\hat{S}).
$$

Since $|\omega_Y|$ is base point free, and it has the same dimension as
$|\omega_{\hat{S}}|$, this implies that $Y$ has at most rational double
points as singularities. This concludes the proof that $S$ is a
primary Burniat surface.

\end{proof}

\section{The Burniat surface with $K^2 =2$ is a classical Campedelli surface}

The aim of this short last section is to illustrate
  how the Burniat surface with $K^2 = 2$ can be seen as a classical
Campedelli surface (with fundamental group $(\ZZ / 2 \ZZ)^3$).

A classical Campedelli surface can be described as the
tautological $(\ZZ / 2 \ZZ)^3$
Galois-covering of $\PP^2$ branched in seven lines.

This means that each line $\{l_{\alpha} = 0\}$  is set to correspond 
to a non zero element of the Galois group
$(\ZZ / 2 \ZZ)^3$, and then, for each character $\chi \in {(\ZZ / 2 \ZZ)^3}^*$,
we consider the covering given (cf. \cite{pardini}) by
$$ w_{\chi} w_{\chi'}  =  \prod_{\chi(\nu)  = \chi'( \nu)=1 } l_{\nu} 
\ w_{\chi+\chi'}$$
in the vector bundle whose sheaf of sections is
$$\bigoplus _{\chi \in {(\ZZ / 2 \ZZ)^3}^* } \hol_{\PP^2} (1) .$$

As we have seen before, the singular model $X$ of a Burniat surface 
$S$ with $K^2_S = 2$
(i.e., $K^2_{X} = 6$) is the
  $(\ZZ / 2 \ZZ)^2$ Galois covering branched in $9$ lines having 4 
points of type $(1,1,1)$,
  whereas the minimal
model $S$ of a Burniat surface with $K_S^2 = 2$ is the smooth 
bidouble cover of a
weak Del Pezzo surface
  $Y''$ of degree $2$.
Note that the strict transforms of the lines of $D \subset \PP^2$ 
passing through 2 of the points
$P_4,P_5,P_6,P_7$ yield rational $(-2)$-curves on $Y''$. There are 
six of them on $Y''$,
namely $D_{i,j}$ for $1
\leq i \leq 3$, $j \in \{2,3 \}$.

Contracting these six $(-2)$ curves, we obtain a  normal Del Pezzo 
surface $Y'$ of degree 2
having six nodes, and with $-K_{Y'}$ ample.

Then the anticanonical map $\varphi:=\varphi_{|-K_{Y'}|} : Y' \ra \PP^2$
is a finite double
cover branched on a quartic curve, which has 6 nodes (since $Y'$ has 
six nodes).

But a plane quartic having $6$ nodes has to be the union of four 
lines $L_1, L_2, L_3, L_4$ in general position.

Let $S \ra Y''$ be the bidouble cover branched in the Burniat 
configuration yielding a minimal model
  of the Burniat surface with $K_S^2 = 2$. Then the preimages of the
$(-2)$-curves $D_{i,j}$ for $1 \leq i \leq 3$, $j
\in \{2,3 \}$ on $Y$ are rational $(-2)$ curves of $S$.

Let now $X'$ be the canonical model of $S$ and consider the 
composition of the bidouble cover
$\psi:X' \ra Y'$ with $\varphi$.

  Since $\psi$ branches on the image $\Delta$
of $D_{1,1} + D_{2,1} +
D_{3,1}$ in $Y'$ (the other 6 lines being contracted),
we see that the branch divisor of $\varphi \circ \psi$ consists
of
$Q:=L_1 + L_2 + L_3 + L_4$ and
the image of $\Delta$ in $\PP^2$.

Looking at the configuration of the lines (cf. figure \ref{configs}), 
we see that
\begin{itemize}
  \item[i)] $D_{1,1}$ intersects $D_{2,2}$, $D_{2,3}$;
\item[ii)] $D_{2,1}$ intersects $D_{3,2}$, $D_{3,3}$;
\item[iii)] $D_{3,1}$ intersects $D_{1,2}$, $D_{1,3}$.
\end{itemize}
Hence the image of $D_{i,1}$ under $\varphi$ has to intersect two 
nodes of the plane
quartic $Q:=L_1 + L_2 + L_3 + L_4$, which implies that, denoting the
image of $D_{i,1}$ under $\varphi$ by
$L_i'$, the branch divisor of the $(\ZZ / 2\ZZ)^3$ Galois-covering
  $\varphi \circ \psi : X \ra \PP^2$ is a
configuration of seven lines $L_1 + L_2 + L_3 + L_4 + L_1' + L_2' + 
L_3'$, where
$L_1, L_2, L_3, L_4$ are four
lines in general position, i.e., form a complete quadrilateral, and
  $L_1', L_2', L_3'$ are the three diagonals.

The covering $\varphi \circ \psi : X \ra \PP^2$ is a Galois covering
with Galois group $(\ZZ / 2\ZZ)^3$.

In fact we already have
as covering transformations the elements of the Galois group $ G'': = 
(\ZZ / 2\ZZ)^2$
of $\psi$. Moreover the involution $ i : Y' \ra Y'$ can be lifted to $X'$ since
$i$ leaves the individual branch curves invariant (as they are 
inverse image of the diagonals
of the quadrilateral), and also the line bundles associated to the covering
of $Y''$ ($Y''$ is simply connected, whence division by 2 is unique 
in $\Pic (Y'')$).

To show that the covering is the tautological one it suffices  to 
verify that for each
non trivial element of
the Galois group its fixed divisor is exactly the inverse image of 
one of the 7 lines in $\PP^2$.

We omit further details since they are contained in the article 
\cite{kulikov} by Kulikov.

The idea there is simply to take the tautological
cover and observe that it factors as a bidouble cover of
$Y'$ branched on the inverse image of the diagonals, each splitting 
into the divisor corresponding
to the line $D_{i,1}$ and the divisor corresponding to $E_{i+2}$.
Whence Kulikov verifies that one gets in this way the  Burniat 
surface  with $K^2=2$.

\begin{rema}
There are other interesting $(\ZZ / 2\ZZ)^3$-Galois covers of the 
plane branched on the
seven lines $L_1, L_2, L_3, L_4,L_1', L_2', L_3' $.

One such is the fibre product $Z$ of the standard bidouble cover
$\PP^2 \ra \PP^2$ branched on the diagonals $L_1', L_2', L_3'$ with 
the double covering
$Y'$ branched on $L_1, L_2, L_3, L_4$.

This gives $Z$ as a double plane branched on
four conics touching in 12 points. $Z$ is a surface with $K^2_Z = 2$, 
$ p_g (Z) = 3$,
whose singularities are precisely 12 points of type $A_3$.
\end{rema}


\bigskip
\noindent
{\bf
Authors' Adresses:}\\
\noindent
I.Bauer, F. Catanese \\
Lehrstuhl Mathematik VIII,\\
Mathematisches
Institut der Universit\"at Bayreuth\\
NW II\\
Universit\"atsstr. 30\\
95447
Bayreuth

\begin{verbatim}
email: ingrid.bauer@uni-bayreuth.de,
        fabrizio.catanese@uni-bayreuth.de
\end{verbatim}
\end{document}